\documentclass[12pt]{article}
\usepackage{mathrsfs}
\usepackage{amsmath}
\usepackage{amsmath,amsthm,amssymb,amscd}
\usepackage{latexsym}
\usepackage[colorlinks,linkcolor=blue,anchorcolor=blue,citecolor=blue,CJKbookmarks=True]{hyperref}
\usepackage[numbers,sort&compress]{natbib}
\usepackage{hypernat}
\usepackage{cases}

\usepackage{geometry}

\allowdisplaybreaks

\newtheorem{theorem}{Theorem}[section]
\newtheorem{corollary}[theorem]{Corollary}
\newtheorem{lemma}[theorem]{Lemma}

\theoremstyle{definition}

\numberwithin{equation}{section}

\begin{document}


\baselineskip=17pt



\title{{\bfseries One sided $a-$idempotent, one sided $a-$equivalent and $SEP$ elements in a ring with involution}}




\author{
   {{\bfseries \small Hua Yao$^{a,}$\footnote
  {Corresponding author. dalarston@126.com; jcweiyz@126.com}, Junchao Wei$^{b}$}}\\
   {\small $^a$ School of Mathematics and Statistics,
   Huanghuai University,}\\{\small
Zhumadian, Henan 463000, P. R. China}
   \\{\small $^b$ School of Mathematical Sciences,
   Yangzhou University,}\\{\small
Yangzhou, Jiangsu 225002, P. R. China}
}


\date{}

\maketitle

\renewcommand{\thefootnote}{}
\footnote{}
\renewcommand{\thefootnote}{\arabic{footnote}}
\setcounter{footnote}{0}

\begin{abstract}
In order to study the properties of $SEP$ elements, we propose the concepts of one sided $a-$idempotent and one sided $a-$equivalent. Under the condition that an element in a ring is both group invertible and $MP-$invertible, some equivalent conditions of such an element to be an $SEP$ element are given based on these two concepts, as will as based on projections and the second and the third power of some products of some elements.
\end{abstract}

\section{Introduction}
\label{Introduction}

An \emph{involution} $a\mapsto a^*$ in a ring $R$ is an anti-isomorphism of degree $2$, that is,
\begin{center}
$(a^*)^*=a$, ~$(a+b)^*=a^*+b^*$, ~$(ab)^*=b^*a^*$.
\end{center}
A ring $R$ with an involution $*$ is called a \emph{$*-$ring}.

Throughout this paper, unless otherwise stated, ring $R$ considered is a $*$-ring which is also associative with an identity.

An element $a \in R$ is said to be \emph{Moore$-$Penrose invertible} (\emph{MP$-$invertible} for short) \cite {Penr} if there exists some $b\in R$ such that the following Penrose
equations hold:
\begin{center}
(1) $aba=a$, ~(2) $bab=b$, ~(3) $ab=(ab)^*$, ~(4) $ba=(ba)^*$.
\end{center}
There is at most one $b$ such that the above conditions hold. We call it the \emph{Moore$-$Penrose inverse} (\emph{MP$-$inverse} for short) of $a$ and denote it by $a^{\dag}$.
The set of all MP$-$invertible elements of $R$ is denoted by $R^{\dag}$.

Following \cite {IsGr}, an element $a\in R$ (unnecessarily a $*$-ring) is said to be \emph{group invertible} if there is some $b\in R$ satisfying the following conditions:
\begin{center}
$aba=a$, ~$bab=b$, ~$ab=ba$.
\end{center}
There is at most one $b$ such that the above conditions hold. We call it the \emph{group inverse} of $a$ and denote it by $a^{\#}$.
The set of all group invertible elements of $R$ is denoted by $R^{\#}$.

Let $a\in R^\#\cap R^{\dag}$. If $a^\#=a^{\dag}$, then $a$ is called an $EP-$element of $R$. The set of all $EP-$ elements of $R$ is denoted by $R^{EP}$. There are lots of interesting properties about $EP-$elements,
which can be seen in some previous literatures \cite{Chen,DND,MD2011,MD20122,MDK,
ZhaoYaoWei2018,
ZhaoYaoWei2020}.

An element $a$ is called a partial isometry if $a\in R^\dag$
and $a^\ast = a^\dag$. The set of all partial isometries in a ring $R$ is denoted by $R^{PI}$.

An element $e\in R$ is called a projection if $e^2=e=e^*$. Denote the set of all projections of a ring $R$ by $PE(R)$. Clearly, $e\in R$ is a projection if and only if $e=ee^*$ if and only if $e=e^*e$.

If a partial isometry $a\in R$ is also an $EP$ element as well, that is $a^\ast = a^\dag=a^\#$, we call it a strongly $EP$ element, $SEP$ element for short. Denote by $R^{SEP}$ the set of all $SEP$ elements.

Recently, the properties of $SEP$ elements are investigated \cite{ZhangWangChenWei,ZhaoWei,ZhaoYaoWei2020}. In this paper, some new properties of $SEP$ elements are studied further. In Section \ref{section2}, some equivalent conditions of an element $a\in R^{\#}\cap R^{\dag}$ to be an $SEP$ element are given based on some other elements being projections. In Section \ref{section3}, modeling on the concept of idempotent, we introduce the concept of one sided $a-$idempotent, including left $a-$idempotent and right $a-$idempotent. Then some equivalent conditions of an element $a\in R^{\#}\cap R^{\dag}$ to be an $SEP$ element are given based on these concepts.
Equivalent conditions based on the second and the third power of some products of some elements for an element to be an $SEP$ element are provided in \ref{section4}. Finally, in Section \ref{section5}, we propose the concept of one sided $a-$equivalent, including left $a-$equivalent and right $a-$equivalent. On this basis, we give some equivalent conditions for an element $a\in R^{\#}\cap R^{\dag}$ to be an $SEP$ element.



\section{Equivalent conditions based on projections for an element to be an $SEP$ element}
\label{section2}

In \cite [Theorem 2.3] {ZhangWangChenWei}, it is indicated that giving $a\in R^{\#}\cap R^{\dag}$, then $a\in R^{SEP}$ if and only if
$a(a^{\#})^*a^{\dag}=a^{\dag}a^2$. This makes us to give the following theorem.

\begin{theorem}
\label{theorem2.1}
Let $a\in R^{\#}\cap R^{\dag}$. Then $a\in R^{SEP}$ if and only if $a(a^{\#})^*a^{\dag}a^{\#}\in PE(R)$.
\end{theorem}

\begin{proof}
$\Rightarrow$ Since $a\in R^{SEP}$, $a(a^{\#})^*a^{\dag}=a^{\dag}a^2$ by \cite [Theorem 2.3] {ZhangWangChenWei}. This gives
$$a(a^{\#})^*a^{\dag}a^{\#}=a^{\dag}a\in PE(R).$$
$\Leftarrow$ From the assumption, one gets
$$a(a^{\#})^*a^{\dag}a^{\#}=(a(a^{\#})^*a^{\dag}a^{\#})^*(a(a^{\#})^*a^{\dag}a^{\#})=(a^{\#})^*(a^{\dag})^*a^{\#}a^*a(a^{\#})^*a^{\dag}a^{\#}.$$
Multiplying the equality on the right by $a^2a^*a^{\dag}a$, one has $a=(a^{\#})^*(a^{\dag})^*a^{\#}a^*a$.
Noting that $(a^{\#})^*=a^{\dag}a(a^{\#})^*$, one obtains $a=a^{\dag}a^2$ and so $a\in R^{EP}$.
It follows that
$$(a^{\dag})^*=aa^{\dag}(a^{\dag})^*=((a^{\#})^*(a^{\dag})^*a^{\#}a^*a)a^{\dag}(a^{\dag})^*=(a^{\#})^*(a^{\dag})^*a^{\#}$$
and
$$(a^{\dag})^*a=(a^{\#})^*(a^{\dag})^*a^{\#}a=(a^{\#})^*(a^{\dag})^*.$$
Applying the involution on the equality, one obtains
$$a^*a^{\dag}=a^{\dag}a^{\#}=a^{\dag}a^{\dag}.$$
By \cite [Corollary 2.10] {ZhaoWei}, $a\in R^{PI}$. Thus $a\in R^{SEP}$.
\end{proof}

\begin{theorem}
\label{theorem2.2}
Let $a\in R^{\#}\cap R^{\dag}$. Then $a\in R^{SEP}$ if and only if $a^{\dag}a(a^{\dag})^*a^{\dag}\in PE(R)$.
\end{theorem}

\begin{proof}
$\Rightarrow$ Suppose that $a\in R^{SEP}$. Then $a(a^{\#})^*a^{\dag}=a^{\dag}a^2$ by \cite [Theorem 2.3] {ZhangWangChenWei} and $a^{\dag}=a^{\#}$. It follows that
$$a^{\dag}a(a^{\dag})^*a^{\dag}=a^{\dag}aa^{\#}a(a^{\#})^*a^{\dag}=a^{\dag}aa^{\#}a^{\dag}a^2=a^{\dag}a\in PE(R).$$

$\Leftarrow$ From the hypothesis, one yields
$$a^{\dag}a(a^{\dag})^*a^{\dag}=a^{\dag}a(a^{\dag})^*a^{\dag}(a^{\dag}a(a^{\dag})^*a^{\dag})^*
=a^{\dag}a(a^{\dag})^*a^{\dag}(a^{\dag})^*a^{\dag}a^{\dag}a.$$
Multiplying the equality on the left by $a^*aa^{\#}$, one gets $a^{\dag}=a^{\dag}(a^{\dag})^*a^{\dag}a^{\dag}a$
and
$$a^*=a^*aa^{\dag}=a^*aa^{\dag}(a^{\dag})^*a^{\dag}a^{\dag}a=a^{\dag}a^{\dag}a.$$
So $a^*a^{\dag}=a^{\dag}a^{\dag}$. By \cite [Corollary 2.10] {ZhaoWei}, $a\in R^{PI}$, and this infers $(a^{\dag})^*=a$.
Now
$$a^{\dag}=a^{\dag}(a^{\dag})^*a^{\dag}a^{\dag}a=a^{\dag}aa^{\dag}a^{\dag}a=a^{\dag}a^{\dag}a.$$
Hence $a\in R^{EP}$ and so $a\in R^{SEP}$.
\end{proof}

From \cite [Theorem 1.5.1] {Mosic}, $a\in R^+$ is a partial isometry if and only if $aa^*\in PE(R)$. This implies the following theorem.

\begin{theorem}
\label{theorem2.3}
Let $a\in R^{\#}\cap R^{\dag}$. Then $a\in R^{SEP}$ if and only if $aa^*a^{\dag}a^{\dag}a^2\in PE(R)$.
\end{theorem}

\begin{proof}
$\Rightarrow$ Since $a\in R^{SEP}$, $a^*=a^{\dag}=a^{\#}$, it follows that
$$aa^*a^{\dag}a^{\dag}a^2=aa^{\#}a^{\#}a^{\#}a^2=aa^{\#}=aa^{\dag}\in PE(R).$$

$\Leftarrow$ From $aa^*a^{\dag}a^{\dag}a^2\in PE(R)$, one gets
$$aa^*a^{\dag}a^{\dag}a^2=(aa^*a^{\dag}a^{\dag}a^2)^2=aa^*a^{\dag}a^{\dag}a^3a^*a^{\dag}a^{\dag}a^2.$$
Multiplying the equality on the left by $(a^{\#})^*a^{\dag}$, one has
$$a^{\dag}a^{\dag}a^2=a^{\dag}a^{\dag}a^3a^*a^{\dag}a^{\dag}a^2.$$
It follows from \cite [Lemma 2.11] {ZhaoWei} that
$$a^{\dag}a^2=a^{\dag}a^3a^*a^{\dag}a^{\dag}a^2$$
and
$$a=aa^{\#}a^{\dag}a^2=aa^{\#}a^{\dag}a^3a^*a^{\dag}a^{\dag}a^2=a^2a^*a^{\dag}a^{\dag}a^2.$$
This infers
$$a^{\#}a=aa^*a^{\dag}a^{\dag}a^2\in PE(R).$$
Hence $a\in R^{EP}$ by \cite [Theorem 1.1.3] {Mosic}. It follows that
$$aa^*=aa^*a^{\dag}a^{\dag}a^2\in PE(R).$$
So $a\in R^{PI}$ and thus $a\in R^{SEP}$.
\end{proof}

Let $a\in R^{\#} \cap R^{\dag}$ and write $\chi_{a}=\{a,a^{\#},a^{\dagger},a^{*}, (a^{\dagger})^{*}, (a^{\#})^{*}\}$. We obtain the following corollary.

\begin{corollary}
\label{corollary2.4}
Let $a\in R^{\#}\cap R^{\dag}$. Then $a\in R^{SEP}$ if and only if $aa^*a^{\dag}xx^{\dag}a\in PE(R)$ for some $x\in \chi_{a}$.
\end{corollary}

\begin{proof}
$\Rightarrow$ Since $a\in R^{SEP}$, $aa^*a^{\dag}a^{\dag}a^2\in PE(R)$ by Theorem \ref{theorem2.3}. Choosing $x=a^{\dag}$, we are done.

$\Leftarrow$ If there exists $x_0\in \chi_a$ such that $aa^*a^{\dag}x_0x_0^{\dag}a\in PE(R)$, then

(1) if $x_0\in \tau_a=\{a, a^{\#}, (a^{\dag})^*\}$, then $x_0x_0^{\dag}=aa^{\dag}$, so
$$aa^*a^{\dag}a=aa^*a^{\dag}aa^{\dag}a=aa^*a^{\dag}x_0x_0^{\dag}a\in PE(R).$$
Similar to the proof of Theorem \ref{theorem2.3}, $a\in R^{SEP}$.

(2) if $x_0\in \gamma_a=\{a^{\dag}, a^*, (a^{\#})^*\}$, then $x_0x_0^{\dag}=a^{\dag}a$, so
$$aa^*a^{\dag}a^{\dag}a^2=aa^*a^{\dag}x_0x_0^{\dag}a\in PE(R).$$
By Theorem \ref{theorem2.3}, $a\in R^{SEP}$.
\end{proof}

\begin{theorem}
\label{theorem2.5}
Let $a\in R^{\#}\cap R^{\dag}$. Then $a\in R^{SEP}$ if and only if $a^{\dag}a^3a^*a^{\dag}\in PE(R)$.
\end{theorem}

\begin{proof}
$\Rightarrow$ Since $a\in R^{SEP}$, $a(a^{\#})^*a^{\dag}=a^{\dag}a^2$ by \cite [Theorem 2.3] {ZhangWangChenWei}. This infers
$$a^{\dag}a^3a^*a^{\dag}=a(a^{\#})^*a^{\dag}aa^*a^{\dag}=aa^{\dag}\in PE(R).$$

$\Leftarrow$ From $a^{\dag}a^3a^*a^{\dag}\in PE(R)$, one gets
$$a^{\dag}a^3a^*a^{\dag}=(a^{\dag}a^3a^*a^{\dag})^*=(a^{\dag})^*aa^*a^*a^{\dag}a
=((a^{\dag})^*aa^*a^*a^{\dag}a)a^{\dag}a=a^{\dag}a^3a^*a^{\dag}a^{\dag}a.$$
Multiplying the equality on the left by $(a^{\#})^*a^{\dag}a^{\#}$, one yields
$a^{\dag}=a^{\dag}a^{\dag}a$. Thus $a\in R^{PE}$, which leads to $a^2a^*a^{\dag}=a^{\dag}a^3a^*a^{\dag}\in PE(R)$.
Hence
$$a^2a^*a^{\dag}=(a^2a^*a^{\dag})^2=a^2a^*a^{\dag}a^2a^*a^{\dag}=a^2a^*aa^*a^{\dag}.$$
Multiplying the equality on the left by $a^{\dag}a^{\#}$ and on the right by $a$, one obtains
$a^*=a^*aa^*$.
This induces $a\in R^{PI}$. Thus $a\in R^{SEP}$.
\end{proof}

Note that $xx^+=\left\{\begin{aligned}aa^+,& \ x \in \tau_a, \\a^+a,& \ x\in \gamma_a.\end{aligned}\right.$ Then Theorem \ref{theorem2.5} induces the following corollary.

\begin{corollary}
\label{corollary2.6}
Let $a\in R^{\#}\cap R^{\dag}$. Then the followings are equivalent:

(1) $a\in R^{SEP}$;

(2) $xx^{\dag}a^2a^*a^{\dag}\in PE(R)$ for some $x\in \gamma_a$;

(3) $x^{\dag}xa^2a^*a^{\dag}\in PE(R)$ for some $x\in \tau_a$.
\end{corollary}

\begin{corollary}
\label{corollary2.7}
Let $a\in R^{\#}\cap R^{\dag}$. Then $a\in R^{SEP}$ if and only if $a^2a^*a^{\#}\in PE(R)$.
\end{corollary}

\begin{proof}
$\Rightarrow$ Assume that $a\in R^{SEP}$. Then $a\in R^{EP}$ and $a^{\dag}a^3a^*a^{\dag}\in PE(R)$ by Theorem \ref{theorem2.5}. This implies $a^2a^*a^{\#}=a^{\dag}a^3a^*a^{\dag}\in PE(R)$.

$\Leftarrow$ From the assumption, we have
$$a^2a^*a^{\#}=(a^2a^*a^{\#})^*=(a^{\#})^*aa^*a^*=(a^{\#})^*aa^*a^*aa^{\dag}=a^2a^*a^{\#}aa^{\dag}.$$
Multiplying the equality on the left by $(a^{\dag})^*a^{\dag}a^{\#}$, one has $a^{\#}=a^{\#}aa^{\dag}$. Hence
$a\in R^{EP}$ by \cite [Theorem 1.2.1] {Mosic}.
Now we have $$a^{\dag}a^3a^*a^{\dag}=a^2a^*a^{\#}\in PE(R).$$ By Theorem \ref{theorem2.5}, $a\in R^{SEP}$.
\end{proof}

\begin{lemma}
\label{lemma2.8}
Let $a\in R^{\#}\cap R^{\dag}$ and $x\in PE(R)$. If $x=aa^+xa^+a$, then $a^+axaa^+\in PE(R)$.
\end{lemma}

\begin{proof}
Since $x\in PE(R)$, $x=x^*$. Then
$$a^+axaa^+=a^+ax^*aa^+=(aa^+xa^+a)^*=x^*=x\in PE(R).$$
\end{proof}

Corollary \ref{corollary2.7} and Lemma \ref{lemma2.8} imply

\begin{corollary}
\label{corollary2.9}
Let $a\in R^{\#}\cap R^{\dag}$. Then $a\in R^{SEP}$ if and only if $a^+a^3a^*a^{\#}aa^+\in PE(R)$.
\end{corollary}

Lemma \ref{lemma2.8} implies if $a^+axaa^+\in PE(R)$ and $x^*=x$, then $aa^+xa^+a\in PE(R)$. Hence Theorem \ref{theorem2.5} infers the following corollary.

\begin{corollary}
\label{corollary2.10}
Let $a\in R^{\#}\cap R^{\dag}$. Then $a\in R^{SEP}$ if and only if $a^2a^*a^\#\in PE(R)$.
\end{corollary}

\section{Equivalent conditions based on $a-$idempotents for an element to be an $SEP$ element}
\label{section3}

Let $e, a\in R$. If $e^2=ae$, then $e$ is called a \emph{left $a-$idempotent}. Similarly, $e$ is called a \emph{right $a-$idempotent} if $e^2=ea$. The following lemma is evident.
\begin{lemma}
\label{lemma3.1}
$e$ is a left $a-$idempotent if and only if $a-e$ is a right $a-$idempotent.
\end{lemma}

\begin{theorem}
\label{theorem3.2}
Let $a\in R^{\#}\cap R^{\dag}$. Then $a\in R^{SEP}$ if and only if $a(a^{\#})^*a^{\dag}$ is a left $a^{\dag}a^2-$idempotent.
\end{theorem}

\begin{proof}
$\Rightarrow$ Since $a\in R^{SEP}$, $a(a^{\#})^*a^{\dag}=a^{\dag}a^2$ by \cite [Theorem 2.3] {ZhangWangChenWei}. So $a(a^{\#})^*a^{\dag}$ is a left $a^{\dag}a^2-$idempotent.

$\Leftarrow$ From the assumption, one has
$(a(a^{\#})^*a^{\dag})^2=a^{\dag}a^2(a(a^{\#})^*a^{\dag})$. Hence
$$a(a^{\#})^*(a^{\#})^*a^{\dag}=a^{\dag}a^3(a^{\#})^*a^{\dag}.$$
Multiplying the equality one the right by $aa^*a^{\dag}$, one gets
$$a(a^{\#})^*a^{\dag}=a^{\dag}a^3a^{\dag}=a^{\dag}a(a^{\dag}a^3a^{\dag})=a^{\dag}a^2(a^{\#})^*a^{\dag}.$$
Multiplying the last equality on the right by $aa^*a^{\dag}a$, one yields $a=a^{\dag}a^2$. Hence $a\in R^{EP}$. It follows that $$a(a^{\#})^*a^{\dag}=a^{\dag}a^3a^{\dag}=a$$
and
$$a^*=a^*a^{\dag}a=a^*a^{\dag}(a(a^{\#})^*a^{\dag})=a^{\dag}.$$
Hence $a\in R^{PI}$ and so $a\in R^{SEP}$.
\end{proof}

Similarly, we have the following theorem by \cite [Theorem 2.3] {ZhangWangChenWei}.

\begin{theorem}
\label{theorem3.3}
Let $a\in R^{\#}\cap R^{\dag}$. Then the followings are equivalent:

(1) $a\in R^{SEP}$;

(2) $a(a^{\#})^*a^{\dag}$ is a right $a^{\dag}a^2-$idempotent;

(3) $a^{\dag}a^2$ if a left $a(a^{\#})^*a^{\dag}-$idempotent;

(4) $a^{\dag}a^2$ is a right $a(a^{\#})^*a^{\dag}-$idempotent.
\end{theorem}

By Lemma \ref{lemma3.1}, Theorems \ref{theorem3.2} and \ref{theorem3.3}, we have the following theorem.

\begin{theorem}
\label{theorem3.4}
Let $a\in R^{\#}\cap R^{\dag}$. Then the following are equivalent:

(1) $a\in R^{SEP}$;

(2) $a(a^{\#})^*a^{\dag}-a^{\dag}a^2$ is a right $-a^{\dag}a^2-$idempotent;

(3) $a(a^{\#})^*a^{\dag}-a^{\dag}a^2$ is a left $-a^{\dag}a^2-$idempotent;

(4) $a^{\dag}a^2-a(a^{\#})^*a^{\dag}$ is a right $-a(a^{\#})^*a^{\dag}-$idempotent;

(5) $a^{\dag}a^2-a(a^{\#})^*a^{\dag}$ is a right $-a(a^{\#})^*a^{\dag}-$idempotent.
\end{theorem}

\begin{theorem}
\label{theorem3.5}
Let $a\in R^{\#}\cap R^{\dag}$. Then $a\in R^{SEP}$ if and only if $a(a^{\#})^*a^{\dag}$ is a right $a-$idempotent.
\end{theorem}

\begin{proof}
$\Rightarrow$ It is an immediate result of Theorem \ref{theorem3.3} because $a^{\dag}a^2=a$.

$\Leftarrow$ From the assumption, we have $(a(a^{\#})^*a^{\dag})^2=a(a^{\#})^*a^{\dag}a$. Thus $$a(a^{\#})^*(a^{\#})^*a^{\dag}=a(a^{\#})^*a^{\dag}a.$$
Multiplying the equality on the left by $a^*a^{\dag}$, one obtains $(a^{\#})^*a^{\dag}=a^{\dag}a$ and
$$a^{\dag}=a^*(a^{\#})^*a^{\dag}=a^*a^{\dag}a.$$
Hence $a\in R^{SEP}$ by \cite [Theorem 1.5.3] {Mosic}.
\end{proof}

Similarly, we have the following theorem.

\begin{theorem}
\label{theorem3.6}
Let $a\in R^{\#}\cap R^{\dag}$. Then $a\in R^{SEP}$ if and only if $a^{\dag}a^2$ is a left $(a^{\dag})^*-$idempotent.
\end{theorem}

\section{Equivalent conditions based on the second and the third power of some products of some elements for an element to be an $SEP$ element}
\label{section4}

\begin{theorem}
\label{theorem4.1}
Let $a\in R^{\#}\cap R^{\dag}$. Then $a\in R^{SEP}$ if and only if $(a(a^{\#})^*a^{\dag})^k=(a^{\dag}a^2)^k$ for $k=2, 3$.
\end{theorem}

\begin{proof}
$\Rightarrow$ It is an immediate result of \cite [Theorem 2.3] {ZhangWangChenWei}.

$\Leftarrow$ From the assumption, one gets
$$a^{\dag}a^4=(a^{\dag}a^2)^3=(a(a^{\#})^*a^{\dag})^3=(a(a^{\#})^*a^{\dag})^2a(a^{\#})^*a^{\dag}
=(a^{\dag}a^2)^2a(a^{\#})^*a^{\dag}=a^{\dag}a^4(a^{\#})^*a^{\dag}$$
and
$$a=a^{\#}a^{\#}a^{\dag}a^4=a^{\#}a^{\#}a^{\dag}a^4(a^{\#})^*a^{\dag}=a(a^{\#})^*a^{\dag}.$$
This gives
$$a^{\dag}a=a^{\dag}a(a^{\#})^*a^{\dag}=(a^{\#})^*a^{\dag}$$
and
$$a^*a^{\dag}a=a^*(a^{\#})^*a^{\dag}=a^{\dag}.$$
By \cite [Theorem 1.5.3] {Mosic}, $a\in R^{SEP}$.
\end{proof}

\begin{theorem}
\label{theorem4.2}
Let $a\in R^{\#}\cap R^{\dag}$. Then $a\in R^{SEP}$ if and only if $(aa^*a^{\dag}a^{\dag}a^2)^k\in PE(R)$ for $k=2, 3$.
\end{theorem}

\begin{proof}
$\Rightarrow$ By the proof of Theorem \ref{theorem2.3}, $aa^*a^{\dag}a^{\dag}a^2=aa^{\dag}$. Thus
$$(aa^*a^{\dag}a^{\dag}a^2)^k=aa^{\dag}\in PE(R)$$
for $k=2, 3$.

$\Leftarrow$ The condition $(aa^*a^{\dag}a^{\dag}a^2)^2\in PE(R)$ implies
\begin{eqnarray*}
(aa^*a^{\dag}a^{\dag}a^2)^2&=&((aa^*a^{\dag}a^{\dag}a^2)^2)^*=(a^*a^{\dag}a(a^{\dag})^*aa^*)^2\\
&=&(a^*a^{\dag}a(a^{\dag})^*aa^*)^2aa^{\dag}
=(aa^*a^{\dag}a^{\dag}a^2)^2aa^{\dag}.
\end{eqnarray*}
Multiplying the equality on the left by $(a^{\#})^*a^{\dag}a^{\#}(aa^{\#})^*a(a^{\#})^*a^{\dag}$, one obtains $$a^{\dag}a^{\dag}a^2=a^{\dag}a^{\dag}a^3a^{\dag},$$
which gives $a^{\dag}a^2=a^{\dag}a^3a^{\dag}$ by \cite [Lemma 2.11] {ZhaoWei}. Multiplying the equality on the left by $a^{\#}$, one has $aa^{\#}=aa^{\dag}$. Hence $a\in R^{EP}$.
This induces $(aa^*)^k=(aa^*a^{\dag}a^{\dag}a^2)^k\in PE(R)$ for $k=2, 3$.
So $(aa^*)^4=(aa^*)^2$ and $(aa^*)^6=(aa^*)^3$. Hence
$$(aa^*)^4=(aa^*)^2(aa^*)^2=(aa^*)^4(aa^*)^2=(aa^*)^6=(aa^*)^3.$$
Thus $(aa^*)^2=(aa^*)^3$. Multiplying the equality on the left by $((a^{\dag})^*a^{\dag})^2$, one gets $aa^{\dag}=aa^*$.
By \cite [Theorem 1.5.1] {Mosic}, $a\in R^{PI}$. Hence $a\in R^{SEP}$.
\end{proof}

\begin{theorem}
\label{theorem4.3}
Let $a\in R^{\#}\cap R^{\dag}$. Then $a\in R^{SEP}$ if and only if $(a(a^{\#})^*a^{\dag})^k$ is a left $(a^{\dag}a^2)^k-$idempotent for $k=2, 3$.
\end{theorem}

\begin{proof}
$\Rightarrow$ Since $a\in R^{SEP}$, $a(a^{\#})^*a^{\dag}=a^{\dag}a^2$. Hence $(a(a^{\#})^*a^{\dag})^k$ is a left $(a^{\dag}a^2)^k-$idempotent for $k=2, 3$.

$\Leftarrow$ Following the assumption, one has
\begin{equation}\label{eq:4.1}
a((a^{\#})^*)^4a^{\dag}=(a(a^{\#})^*a^{\dag})^4=(a^{\dag}a^2)^2(a(a^{\#})^*a^{\dag})^2=a^{\dag}a^4(a^{\#})^*(a^{\#})^*a^{\dag}.
\end{equation}
and
\begin{equation}\label{eq:4.2}
a((a^{\#})^*)^6a^{\dag}=(a(a^{\#})^*a^{\dag})^6=(a^{\dag}a^2)^3(a(a^{\#})^*a^{\dag})^3=a^{\dag}a^5((a^{\#})^*)^3a^{\dag}.
\end{equation}
Multiplying (\ref{eq:4.1}) on the left by $a^{\dag}a$, one gets
\begin{equation}\label{eq:4.3}
a((a^{\#})^*)^4a^{\dag}=a^{\dag}a^2((a^{\#})^*)^4a^{\dag}.
\end{equation}
Multiplying (\ref{eq:4.3}) on the right by $a(a^*)^4a^{\dag}$, one obtains
\begin{equation}\label{eq:4.4}
aa^{\dag}=a^{\dag}a^2a^{\dag}.
\end{equation}
So $a\in R^{EP}$ and (\ref{eq:4.1}) and (\ref{eq:4.2}) become into the following equalities.
\begin{equation}\label{eq:4.5}
a((a^{\#})^*)^4a^{\dag}=a^3(a^{\#})^*(a^{\#})^*a^{\dag}.
\end{equation}
\begin{equation}\label{eq:4.6}
a((a^{\#})^*)^6a^{\dag}=a^4((a^{\#})^*)^3a^{\dag}.
\end{equation}
Multiplying (\ref{eq:4.5}) and (\ref{eq:4.6}) on the left by $a^{\dag}$ and on the right by $a$, one has
$$((a^{\#})^*)^4=a^2(a^{\#})^*(a^{\#})^*$$
and
$$((a^{\#})^*)^6=a^3((a^{\#})^*)^3.$$
Then
$$a^2((a^{\#})^*)^4=a^2(a^{\#})^*(a^{\#})^*((a^{\#})^*)^2=((a^{\#})^*)^4((a^{\#})^*)^2=((a^{\#})^*)^6=a^3((a^{\#})^*)^3.$$
Multiplying the above equality on the left by $(a^{\#})^2$ and on the right by $(a^*)^3$, and noting that $a\in R^{EP}$, one gets
$(a^{\#})^*=a$. So $a\in R^{SEP}$.
\end{proof}

\section{Equivalent conditions based on $a-$equivalent relation for an element to be an $SEP$ element}
\label{section5}

Let $a, b, c\in R$. If $ab=ac$, then $b$ and $c$ are said to be \emph{left $a-$equivalent}. Correspondingly, if $ba=ca$, then $b$ and $c$ are said to
be \emph{right $a-$equivalent}.

\begin{theorem}
\label{theorem5.1}
Let $a\in R^{\#}\cap R^{\dag}$. Then $a\in R^{SEP}$ if and only if $a(a^{\#})^*a^{\dag}$ and $a^{\dag}a^2$ are left $a-$equivalent.
\end{theorem}

\begin{proof}
$\Rightarrow$ Since $a\in R^{SEP}$, $a(a^{\#})^*a^{\dag}=a^{\dag}a^2$ by \cite [Theorem 2.3] {ZhangWangChenWei}. So $a(a^{\#})^*a^{\dag}$ and $a^{\dag}a^2$ are left $a-$equivalent.

$\Leftarrow$ From the assumption, we have $a^2(a^{\#})^*a^{\dag}=aa^{\dag}a^2=a^2$.
Multiplying the equality on the left by $a^*a^{\dag}a^{\#}$, one gets $a^{\dag}=a^*a^{\dag}a$.
Hence $a\in R^{SEP}$ by \cite [Theorem 1.5.3] {Mosic}.
\end{proof}

\begin{theorem}
\label{theorem5.2}
Let $a\in R^{\#}\cap R^{\dag}$. Then $a\in R^{SEP}$ if and only if $a(a^{\#})^*a^{\dag}$ and $a^{\dag}a^2$ are left $x-$equivalent for some
$x\in \rho_a=\{a, a^{\#}, a^{\dag}, a^*, (a^{\dag})^*, (a^{\#})^*, (a^{\#})^{\dag}, (a^{\dag})^{\#}\}$.
\end{theorem}

\begin{proof}
$\Rightarrow$ By Theorem \ref{theorem5.1}, $a(a^{\#})^*a^{\dag}$ and $a^{\dag}a^2$ are left $a-$equivalent. Choosing $x=a$, we are done.

$\Leftarrow$ From the assumption, there exists $x\in \rho_a$ such that $xa(a^{\#})^*a^{\dag}=xa^{\dag}a^2$. If $x\in \tau_a$, then $x^{\dag}x=a^{\dag}a$. It follows that
\begin{eqnarray*}
(a^{\#})^*a^{\dag}&=&a^{\dag}a(a^{\#})^*a^{\dag}=a^{\dag}a^{\#}aa^{\dag}a^2(a^{\#})^*a^{\dag}
=a^{\dag}a^{\#}a(x^{\dag}x)a(a^{\#})^*a^{\dag}\\
&=&a^{\dag}a^{\#}ax^{\dag}xa^{\dag}a^2=a^{\dag}a^{\#}aa^{\dag}aa^{\dag}a^2=a^{\dag}a^{\#}a^2=a^{\dag}a.
\end{eqnarray*}
Hence
$$a^*a^{\dag}a=a^*(a^{\#})^*a^{\dag}=a^{\dag}.$$
By \cite [Theorem 1.5.3] {Mosic}, $a\in R^{SEP}$.

If $x\in \gamma_a$, then $x^{\dag}x=aa^{\dag}$. It follows that
\begin{eqnarray*}
(a^{\#})^*a^{\dag}&=&a^{\dag}a(a^{\#})^*a^{\dag}=a^{\dag}aa^{\dag}a(a^{\#})^*a^{\dag}=a^{\dag}x^{\dag}xa(a^{\#})^*a^{\dag}\\
&=&a^{\dag}x^{\dag}xa^{\dag}a^2=a^{\dag}aa^{\dag}a^{\dag}a^2=a^{\dag}a^{\dag}a^2
\end{eqnarray*}
and
$$a^{\dag}=a^*(a^{\#})^*a^{\dag}=a^*a^{\dag}a^{\dag}a^2=(a^*a^{\dag}a^{\dag}a^2)a^{\dag}a=a^{\dag}a^{\dag}a.$$
Thus $a\in R^{EP}$. This induces $(a^{\#})^*a^{\dag}=a^{\dag}a^{\dag}a^2=a^{\dag}a$. Hence $a\in R^{SEP}$.

If $x=(a^{\#})^{\dag}=a^{\dag}a^3a^{\dag}$, then one gets
$a^{\dag}a^3a^{\dag}a(a^{\#})^*a^{\dag}=a^{\dag}a^3a^{\dag}a^{\dag}a^2$. Thus
$$a^{\dag}a^3(a^{\#})^*a^{\dag}=a^{\dag}a^3a^{\dag}a^{\dag}a^2.$$
Multiplying the equality on the left by $a^{\dag}a^{\#}$, one has $(a^{\#})^*a^{\dag}=a^{\dag}a^{\dag}a^2$. Hence $a\in R^{SEP}$ by the above proof.

If $x=(a^{\dag})^{\#}=(aa^{\#})^*a(aa^{\#})^*$, then
$$(aa^{\#})^*a(aa^{\#})^*a(a^{\#})^*a^{\dag}=(aa^{\#})^*a(aa^{\#})^*a^{\dag}a^2=(aa^{\#})^*a^2.$$
Multiplying the equality on the left by $a^{\dag}a^{\dag}$, one obtains $(a^{\#})^*a^{\dag}=a^{\dag}a^{\dag}a^2$. Hence $a\in R^{SEP}$.
\end{proof}

\begin{theorem}
\label{theorem5.3}
Let $a\in R^{\#}\cap R^{\dag}$. Then $a\in R^{SEP}$ if and only if $a(a^{\#})^*a^{\dag}$ and $a^{\dag}a^2$ are left $a^{\dag}a^2-$equivalent .
\end{theorem}

\begin{proof}
$\Rightarrow$ It is clear because $a(a^{\#})^*a^{\dag}=a^{\dag}a^2$.

$\Leftarrow$ From the hypothesis, we have
$$a^{\dag}a^3(a^{\#})^*a^{\dag}=a^{\dag}a^2a^{\dag}a^2=a^{\dag}a^3.$$
Multiplying the equality on the left by $a^{\#}$, one obtain
$a(a^{\#})^*a^{\dag}=a$ and
$$a^{\dag}a=a^{\dag}a(a^{\#})^*a^{\dag}=(a^{\#})^*a^{\dag}.$$
Hence $a\in R^{SEP}$.
\end{proof}

\begin{theorem}
\label{theorem5.4}
Let $a\in R^{\#}\cap R^{\dag}$. Then $a\in R^{SEP}$ if and only if $a(a^{\#})^*a^{\dag}$ and $a^{\dag}a^2$ are left $a^{\dag}aa^{\#}-$equivalent .
\end{theorem}

\begin{proof}
$\Rightarrow$ It is clear by \cite [Theorem 2.3] {ZhangWangChenWei}.

$\Leftarrow$ From the assumption, we have
$$a^{\dag}aa^{\#}a(a^{\#})^*a^{\dag}=a^{\dag}aa^{\#}a^{\dag}a^2=a^{\dag}a,$$
$$a=aa^{\dag}a=a(a^{\dag}aa^{\#}a(a^{\#})^*a^{\dag})=a(a^{\#})^*a^{\dag}.$$
Hence $a\in R^{SEP}$ by the proof of Theorem \ref{theorem5.3}.
\end{proof}

\subsection*{Acknowledgements}
This work is in part supported by the Natural Science Foundation of Henan Province of China under Grant No. 222300420499.

\end{document}